\documentclass[12pt]{amsart}

\usepackage{amsmath,amssymb,mathtools}
\usepackage{graphicx}
\usepackage{xcolor}
\usepackage{tikz}
\usepackage{float}
\usetikzlibrary{arrows.meta}
\usepackage[colorlinks=true,citecolor=blue,linkcolor=blue,urlcolor=blue]{hyperref}
\usepackage{geometry}
\newtheorem{theorem}{Theorem}[section]
\newtheorem{maintheorem}{Theorem}

\newtheorem{maintheoremb}{Theorem}

\newtheorem{proposition}[theorem]{Proposition}
\newtheorem{lemma}[theorem]{Lemma}
\newtheorem{corollary}[theorem]{Corollary}
\theoremstyle{definition}
\newtheorem{definition}[theorem]{Definition}
\theoremstyle{remark}

\newtheorem*{remark*}{Remark}

\numberwithin{equation}{section}
\allowdisplaybreaks[4]
\newcommand{\E}{\mathbb E}
\newcommand{\one}{\mathbf 1}
\begin{document}

\title{Phase transition in optimal hypercontractivity}

\date{}

\author{Jie Cao}
\address{School of Mathematical Sciences, Shenzhen University,
Shenzhen 518060, China}
\email{mathcj@foxmail.com}

\author{Shilei Fan}
\address{School of Mathematics and Statistics, and Key Lab NAA--MOE,
Central China Normal University, Wuhan 430079, China}
\email{slfan@ccnu.edu.cn}

\author{Yong Han}
\address{School of Mathematical Sciences, Shenzhen University,
Shenzhen 518060, China}
\email{hanyong@szu.edu.cn}

\author{Yanqi Qiu}
%\thanks{Corresponding author: Yanqi Qiu,  yanqiqiu@ucas.ac.cn}
\address{School of Fundamental Physics and Mathematical Sciences,
HIAS, University of Chinese Academy of Sciences,
Hangzhou 310024, China}
\email{yanqi.qiu@hotmail.com; yanqiqiu@ucas.ac.cn}

\author{Zipeng Wang}
\address{College of Mathematics and Statistics, Chongqing University,
Chongqing 401331, China}
\email{zipengwang2012@gmail.com; zipengwang@cqu.edu.cn}

\begin{abstract}
We  discover an exponent-dependent phase transition phenomenon for optimal hypercontractivity: for every prescribed
$q_0>2$, there exists a reversible continuous-time Markov chain on three state space with normalized spectral gap whose $(2,q)$-optimal  hypercontractivity time satisfies $$ \text{$t_{\mathrm{opt}}(2,q)=\frac12\log(q-1)$ if and only if $q\ge q_0$},$$
whereas the strict inequality  $t_{\mathrm{opt}}(2,q)>\frac12\log(q-1)$ holds  for $2<q<q_0$.
\end{abstract}


\maketitle

\begingroup

\section{Introduction}

\subsection{Motivations and main results}
Let $Q$ be the Markov generator of an irreducible reversible continuous-time Markov
chain on a finite state space $S$, with invariant probability measure $\mu$,
and let $(P_t=e^{tQ})_{t\geq 0}$ be the corresponding Markov semigroup. Scaling the time if necessary, we  may assume throughout that the spectral gap (the second eigenvalue of $-Q$) is normalized by $\lambda(Q)=1$. For $1<p<q<\infty$, define the
optimal hypercontractive time by
$$
t_{\mathrm{opt}}^{(Q)}(p,q)
:=
\inf\bigl\{
t\geq 0:\|P_t\|_{L^p(\mu)\to L^q(\mu)}\leq 1
\bigr\}.
$$
By the standard small perturbations around constant functions gives the following  universal lower bound:
\begin{equation}\label{eq:spectral-lower-bound}
t_{\mathrm{opt}}^{(Q)}(p,q)
\geq
\tau(p,q)
:=
\frac12\log\frac{q-1}{p-1}.
\end{equation}
Further  backgrounds on finite Markov chains can be found in
\cite{Anderson1991,Kelly1979,Norris1997}.

On  the  symmetric two-point space, the classical inequalities of Bonami, Nelson, Beckner, and Gross \cite{Bonami1970,Nelson1973,Beckner1975,Gross1975} shows that equalities \eqref{eq:spectral-lower-bound} simultaneously hold for all exponent pairs. This formula is a natural target for optimal hypercontractivity in
many other settings, see \cite{Andersson2001,Andersson2002, JungePalazuelosParcetPerrinRicard2015,RicardXu2016,
JungePalazuelosParcetPerrin2017,Yao2025,FrankIvanisvili2026, XieZhangCycles2026,XieZhangFreeGroup2026,Yao2026,ZhangFreeOrthogonal2026}.

The main question in this paper is to study whether sharpness in \eqref{eq:spectral-lower-bound} can depend
on the exponent pair. Namely, we want to show the possibility of  the phase transition phenomenon stated as in the abstract.    

Note that,  if equality holds  for even one component pair in \eqref{eq:spectral-lower-bound}, the spectral-gap eigenspace must have the third-moment-vanishing (TMV) property (see Section~\ref{sec:three-state-family} for the precise definition and reduction). Therefore, for finding the phase transition, we shall restrict our attention to  the class of the Markov generators with TMV property.  Since two-point chains are too rigid, it is natural to consider three-states chains with   TMV property. The Markov generators with TMV property on three-state space  can be fully classified using the method in Section~\ref{sec:three-state-family}. However, for our main purpose, we only restrict to the simplest ones as shown in Figure~\ref{fig:reflection-symmetric-path}. These Markov chains has been 
studied by Chen, Liu, and Saloff-Coste in \cite{ChenLiuSaloffCoste2008}.

\begin{figure}[htbp]
\centering
\includegraphics[scale=1.2]{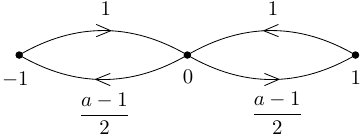}
\caption{The reflection-symmetric three-point Markov chain with generator
$Q_a$. Each label gives the corresponding jump rate.}
\label{fig:reflection-symmetric-path}
\end{figure}

More precisely,  we shall consider  the family $\{Q_a: a>1\}$ of Markov generators,  which, with respect to the order $\{0,-1,1\}$, are given by
\[
Q_a=
\begin{pmatrix}
-(a-1) & (a-1)/2 & (a-1)/2\\
1 & -1 & 0\\
1 & 0 & -1
\end{pmatrix}.
\]
For each $a>1$. The corresponding Markov chain is irreducible and reversible with invariant measure
$\mu_a(0)=1/a$ and $\mu_a(\pm1)=(a-1)/(2a)$. Its spectrum is
$\{0,-1,-a\}$, and hence $\lambda(Q_a)=1$.

Gross identified the full family of sharp hypercontractive inequalities
with the sharp logarithmic Sobolev constant. More precisely,  set 
\[
\rho_{\mathrm{LS}}(Q)
:=      \inf_{\operatorname{Ent}_\mu(f^2)>0}    \frac{2   \langle f,-Qf\rangle_{L^2(\mu)}}{ \operatorname{Ent}_\mu(f^2)} \text{~with~}\operatorname{Ent}_\mu(f^2)=\!\int\! f^2\log f^2\,d\mu
-\! \int \! f^2\,d\mu\cdot \log \! \int \! f^2 \,d\mu.
\]
Then under the assumption $\lambda(Q)=1$, one has
\(\rho_{\mathrm{LS}}(Q)\le1\), and the Gross correspondence states that 
\[
\rho_{\mathrm{LS}}(Q)=1 \Longleftrightarrow  t_{\mathrm{opt}}^{(Q)}(p,q)
 =  \tau(p,q) \, \, \text{for any $1<p<q<\infty$.}
\]

For $1<a\leq2$, Chen, Liu, and Saloff-Coste \cite{ChenLiuSaloffCoste2008} proved that its logarithmic
Sobolev constant attains the extremal value allowed by the spectral gap
precisely when $7/4\leq a\leq2$.  Then, by Gross's correspondence,  the sharp hypercontractive relation holds simultaneously
for all exponent pairs.  See Gross~\cite[Theorems~1 and~2]{Gross1975},
Diaconis and Saloff-Coste~\cite[Theorem~3.5]{DiaconisSaloffCoste1996}, and
Chen, Liu, and Saloff-Coste~\cite[Theorem~1.2]{ChenLiuSaloffCoste2008} for more backgrounds.

Their result determines precisely when the spectral-gap bound is attained
simultaneously for every exponent pair, and naturally leads to a finer
question: once this simultaneous sharpness fails, must the bound fail to be
sharp for every exponent pair, or can sharpness persist on only part of the
exponent range?
\begin{definition}
We say that a Markov chain exhibits a hypercontractivity phase transition if
\[
t_{\mathrm{opt}}^{(Q)}(p,q)=\tau(p,q)
\]
for at least one, but not for every, pair of exponents
$1<p<q<\infty$. Equivalently, we say that its generator $Q$ has a
hypercontractivity phase transition.
\end{definition}

Our first main result gives a sharp affirmative answer to the existence of hypercontractivity phase transition and shows, moreover,
that the transition point can be prescribed.

\begin{maintheorem}\label{thm:prescribed-local-threshold}
For every $q_0>2$, there exists an irreducible reversible continuous-time
Markov chain on a three-state space with spectral gap one such that
\[
t_{\mathrm{opt}}^{(Q)}(2,q)=\tau(2,q)=\frac12\log(q-1)
\]
if and only if $q\geq q_0$. Equivalently,
\[
t_{\mathrm{opt}}^{(Q)}(2,q)>\tau(2,q)=\frac12\log(q-1)
\]
if and only if $2<q<q_0$.
\end{maintheorem}

The phenomenon in Theorem~\ref{thm:prescribed-local-threshold} has no
two-point counterpart, and there is a simple structural reason why it 
becomes possible on three points. At a structural level, the
mean-zero subspace of a two-point space is one-dimensional, and hence there
is only one nonconstant spectral mode. On a three-point space, by contrast,
the mean-zero subspace is two-dimensional. In the family above the two
nonconstant modes have eigenvalues $-1$ and $-a$, and their interaction can
produce different hypercontractive behavior for different exponent pairs.

Theorem~\ref{thm:prescribed-local-threshold} follows from a complete phase
diagram for the family $(Q_a)_{a>1}$. Our analysis reveals three regimes. For
$1<a\leq3/2$, the spectral-gap bound is never attained. For
$3/2<a<7/4$, sharpness depends on the exponent pair and is separated from
strict inequality by a unique critical exponent. For $a\geq7/4$, sharpness
holds for every exponent pair.

\begin{maintheoremb}
\label{thm:family-phase-diagram}
Let $a>1$. Then the generator   $Q_a$ has a
hypercontractivity phase transition if and only if
$a\in (3/2, 7/4)$. More precisely,  for  any $3/2<a<7/4$, there is a unique $p_a\in(1,2)$ such that for every $1<p<2$,
\[
t_{\mathrm{opt}}^{(Q_a)}(p,2)
=\frac12\log\frac1{p-1}
\quad\Longleftrightarrow\quad 1<p\le p_a,
\]
whereas
\[
t_{\mathrm{opt}}^{(Q_a)}(p,2)
>\frac12\log\frac1{p-1}
\quad\Longleftrightarrow\quad p_a<p<2.
\]
 Moreover,  $p_a \in (1,2)$ is the  unique solution of  the equation
\begin{align}\label{pa-eq}
(p_a-1)^{a-1}
=\frac{2ap_a-3}
       {2ap_a-3(p_a-1)}.
\end{align}
 In particular,  the map  $a\mapsto p_a$ is one-to-one from $(3/2, 7/4)$ to $(1,2)$. 
\end{maintheoremb}

For the generators $Q_a$ with the other parameters $a\in (1, \infty) \setminus (3/2, 7/4)$, we have the following 

\begin{proposition}[See 
Corollary~\ref{cor-non-classical} and
Proposition~\ref{thm:global-rigidity}]
\label{prop:family-phase-diagram}
Let $a\in (1, \infty) \setminus (3/2, 7/4)$. Then 
\begin{itemize}
\item If $1<a\le 3/2$, then, for every $1<p<2$,
\[
t_{\mathrm{opt}}^{(Q_a)}(p,2)
>\frac12\log\frac1{p-1}.
\]

\item If $a\ge 7/4$, then $Q_a$ is hypercontractively rigid; in
particular, for every $1<p<2$,
\[
t_{\mathrm{opt}}^{(Q_a)}(p,2)
=\frac12\log\frac1{p-1}.
\]
\end{itemize}
\end{proposition}
\endgroup

It is worthwhile to mention that there are two different types of phase transitions in the family $(Q_a)_{a>1}$ (the second compares different Markov chains, while the first concerns only a fixed one):
\begin{itemize}
\item \emph{Theorem~\ref{thm:family-phase-diagram}: the exponent phase transition.}  Fix
$a\in(3/2,7/4)$, hence fix one Markov chain, and vary the exponent
pair. Equality in the universal lower bound then holds on one side
of the phase-transition curve and fails on the other.  This phase transition is the main focus of this paper.
\item \emph{Theorem~\ref{thm:family-phase-diagram} and Proposition~\ref{prop:family-phase-diagram}: the parameter phase transition.}   For $a\le3/2$, the strict inequality  $t_{\mathrm{opt}}^{(Q_a)}(p,2)>\tau(p,2)$ holds for 
$1<p<2$; hypercontractivity phase transition occurs for
$3/2<a<7/4$; and for $a\ge7/4$ hypercontractive rigidity holds.
\end{itemize}

The phase-transition curve \eqref{pa-eq}, which joins $(3/2,1)$ and $(7/4,2)$, is determined by the vanishing of the discriminant 
\[
\Delta(a,p):=
 \frac{(2-p)(2ap-3p+3)}{3}
\left(
(p-1)^{a-1}
-\frac{2ap-3}{2ap-3p+3}
\right)
\] of the quadratic form \eqref{quad-form} below. 
For each $a\in(3/2,7/4)$, equality
$t_{\mathrm{opt}}^{(Q_a)}(p,2)=\tau(p,2)$ holds on and below this curve, while the
inequality is strict above it.
See Figure \ref{fig:implicit-phase-diagram}.

\begin{figure}[htbp]
\centering
\includegraphics[scale=1.2]{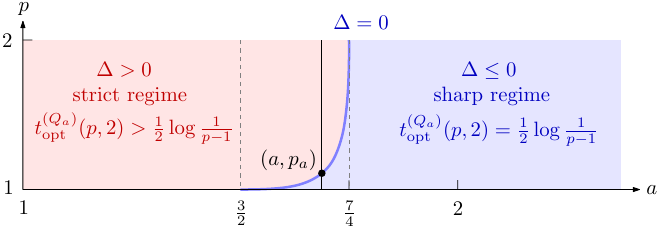}
\caption{The parameter--exponent phase diagram for the family $Q_a$.
The blue region, including the transition curve, is the sharp regime
$t_{\mathrm{opt}}^{(Q_a)}(p,2)=\tau(p,2)$; the red region is the strict
regime.  In particular, the entire region to the right of $a=7/4$ is
blue because $Q_a$ is hypercontractively rigid there.}
\label{fig:implicit-phase-diagram}
\end{figure}

\begin{remark*}
The optimal
constants for biased Bernoulli variables rule out the hypercontractivity phase transition on
two states (see
Oleszkiewicz~\cite{Oleszkiewicz2003}, Wolff~\cite{Wolff2007} and our previous work \cite{CaoFanHanQiuWang2026}).  Thus three states are the
first setting in which a fixed Markov chain can exhibit this phase
transition.
\end{remark*}

\subsection{Comparison with results in literature}
The weighted three-point path introduced in
Figure~\ref{fig:reflection-symmetric-path} is a very simple
birth-and-death Markov chain.  Closely related three-point sticks
have appeared in the study of logarithmic Sobolev inequalities; see
Chen, Liu, and Saloff-Coste~\cite[Section~4]{ChenLiuSaloffCoste2008}
and Faust and Fawzi~\cite[Section~V-E]{FaustFawzi2024}.  The
comparison with \cite[Theorem~4.3]{ChenLiuSaloffCoste2008} is described as follows.
If $\theta$ denotes the parameter used in \cite[Theorem~4.3]{ChenLiuSaloffCoste2008}, then, after relabelling
the three states,
\[
Q_a=K_\theta-I,\qquad \theta=a-1,
\]
and the invariant measures agree.  Since their spectral gap is one
and our convention gives $\rho_{\mathrm{LS}}=2\alpha_\theta$, their
criterion
\[
\alpha_\theta=1/2
\quad\Longleftrightarrow\quad
\theta\in[3/4,1]
\]
is precisely hypercontractive rigidity for $7/4\le a\le2$.
Their restriction $\theta\le1$, or equivalently $a\le2$, is required
because $K_\theta$ is a discrete-time Markov kernel.  By contrast,
$Q_a$ remains a continuous-time Markov generator for every $a>1$.
In particular, when $a>2$, the matrix $I+Q_a$ has the negative holding entry
$2-a$ and hence is no longer a Markov kernel, although
$e^{tQ_a}$ is a Markov kernel for every $t\ge0$.

Theorem~\ref{thm:prescribed-local-threshold} and
Chen, Liu, and Saloff-Coste's
\cite[Theorem~4.3]{ChenLiuSaloffCoste2008} have different but largely
complementary emphases. Chen, Liu, and Saloff-Coste determine when the optimal
logarithmic Sobolev constant reaches one half of the spectral gap.
Our main concern is instead the interval $3/2<a<7/4$.  In this
interval their theorem implies $\alpha_\theta<1/2$, but it does not
distinguish the exponent pairs for which the universal
hypercontractive lower bound is attained from those for which the
inequality is strict.  Theorem~\ref{thm:prescribed-local-threshold}
resolves precisely this pairwise question and reveals the
hypercontractivity phase transition.  Thus this apparently simple
Markov chain has a substantially richer hypercontractive behavior
than its logarithmic Sobolev constant alone records.

\section{The third-moment-vanishing property and three-state TMV chains}
\label{sec:three-state-family}

\subsection{The TMV condition}

The key notion in our construction is the
\emph{third-moment-vanishing property (abbreviated by TMV)}, defined
as follows.

Let
\[
E_{-1}(Q):=\ker(Q+I)\subset L^2(\mu),
\]
where $Q$ is regarded as an operator on $L^2(\mu)$.  Thus
$E_{-1}(Q)$ is the eigenspace associated with the eigenvalue $-1$.

\begin{definition}
\label{def:TMV}
The Markov chain is said to have the \emph{$\mathrm{TMV}$ property} if
$$
\int_S\varphi^3\,d\mu=0
\qquad\text{for every }\varphi\in E_{-1}(Q).
$$
Equivalently, we say that the generator $Q$, or the Markov semigroup
$(P_t)_{t\ge0}$ generated by $Q$, has the TMV property.
\end{definition}

The TMV condition is introduced precisely because of the following
necessity statement: if
$t_{\mathrm{opt}}^{(Q)}(p,q)=\tau(p,q)$ for even one pair
$1<p<q<\infty$, then $Q$ is TMV.  This follows by expanding
$\|\one+\varepsilon\varphi\|_r$ at the constant function: at
$t=\tau(p,q)$ the quadratic terms agree, and comparison for both
signs of $\varepsilon$ forces
$\int_S\varphi^3\,d\mu=0$.

\begin{remark*}
The TMV condition is closely related to the perturbative obstruction in
Chen, Liu, and Saloff-Coste~\cite[Theorem~1.1]{ChenLiuSaloffCoste2008}:
there, a Taylor expansion at constant functions shows that  $\rho_{\mathrm{LS}}(Q)= \lambda(Q)$ implies that $Q$ is TMV.  An earlier example distinguishing
spectral and logarithmic Sobolev information appears in
Korzeniowski and Stroock~\cite{KorzeniowskiStroock1985}.
\end{remark*}

\begin{remark*}
For Bernoulli masses $\alpha$ and $1-\alpha$, with $0<\alpha<1$, the
third moment of a nonzero eigenfunction in $E_{-1}(Q)$ vanishes only
in the uniform case $\alpha=1/2$.  The sharp two-point inequality then yields equality in
\eqref{eq:spectral-lower-bound} for all exponent pairs in the uniform
case and for none in the biased case.  Hence no two-state Markov chain
has a hypercontractivity phase transition.  For the classical
foundations and later two-point results, see
\cite{Bonami1970,Nelson1973,Beckner1975,Wolff2007,
Oleszkiewicz2003}.
\end{remark*}

We now specialize the TMV condition to a three-state space.    Our first goal is to classify the possible third-moment
eigenfunctions and the corresponding reversible generators.

Throughout this section the state space is
\[
S=\{-1,0,1\},
\]
written in the coordinate order $(0,-1,1)$:  thus, a vector
$(x,y,z)^\top$ represents the function $f$ with
$f(0)=x$, $f(-1)=y$, and $f(1)=z$: 
$$
(x,y,z)^\top \longleftrightarrow   f(0)=x, \quad  f(-1)=y, \quad f(1)=z. 
$$
 This symmetric labeling is
convenient because the eigenfunction corresponding to the eigenvalue
$-1$ in our construction is the odd function
$e_1=(0,1,-1)^\top$.

For our purposes it is enough to consider invariant distributions
that are symmetric under the reflection $x\mapsto-x$ on $S$.  Thus, for
simplicity, we fix
$$
\nu_c(0)=1-c,
\qquad
\nu_c(-1)=\nu_c(1)= c/2,
\qquad 0<c<1.
$$

Our classification for the invariant measure $\nu_c$ proceeds in
three steps:
\begin{itemize}
\item First, we start with a finite prescribed-moment problem of solving
\begin{equation}\label{eq:three-state-moment-system}
\E_{\nu_c} h=0,
\qquad
\E_{\nu_c} h^3=0,
\end{equation}
and obtain a complete list, up to scaling and reflection, of the
possible functions $h$.
\item      Second, we restrict to the odd solution and classify the corresponding
reversible generators for which the odd solution $h$ spans the
eigenspace corresponding to the eigenvalue $-1$:
\[
E_{-1}(Q)=\operatorname{span}\{h\}.
\]
\item Third, starting from the constant eigenfunction $e_0=\one$ and the
chosen second eigenfunction $e_1=h$, we find a third vector $e_2$
orthogonal to both in $L^2(\nu_c)$.  Assigning the eigenvalues
$0,-1,-a$ to $e_0,e_1,e_2$ determines
\[
Q=\mathsf E\operatorname{diag}(0,-1,-a)\mathsf E^{-1},
\qquad
\mathsf E=[e_0,e_1,e_2],
\]
where $\mathsf E$ is the $3\times3$ matrix with columns
$e_0,e_1,e_2$,
and the requirement that the off-diagonal entries of $Q$ be
nonnegative gives exactly the admissible generators.
\end{itemize}

\begin{remark*}
Unlike the classical moment problem, in the finite prescribed-moment
problem \eqref{eq:three-state-moment-system} the measure $\nu_c$ is
fixed and the unknown is a random variable; see
\cite{Schmudgen2017} for general background.
\end{remark*}

\subsection{Moment classification}

\begin{proposition}
\label{prop:third-moment-functions}
Up to multiplication by a nonzero scalar, every nonzero function
$h$ satisfying
\[
\E_{\nu_c}h=0,
\qquad
\E_{\nu_c}h^3=0
\]
is of one of the following forms:
\begin{enumerate}
\item the odd solution
\[
h_{\mathrm{odd}}=(0,1,-1)^\top;
\]
\item if $c\ge1/2$, the two non-odd solutions
\[
h_{\mathrm{nonodd}}^{\pm}
=\left(
1,\frac{c-1\pm\sqrt{(2c-1)/3}}{c},
  \frac{c-1\mp\sqrt{(2c-1)/3}}{c}
\right)^\top.
\]
They are interchanged by reflection of the last two coordinates and
coincide when $c=1/2$.
\end{enumerate}
\end{proposition}

\begin{proof}
Write $h=(h_0,h_-,h_+)^\top$.  If $h_0=0$, the mean-zero condition
gives $h_-+h_+=0$, hence the odd solution.

Suppose $h_0\ne0$ and normalize $h_0=1$.  Write $h_-=u$ and $h_+=v$.
The mean-zero condition is
\[
u+v=\delta:=2(c-1)/c.
\]
The cubic condition becomes
$
u^3+v^3=\delta.
$
Using $u^3+v^3=(u+v)^3-3uv(u+v)$ gives
\[
uv=(\delta^2-1)/3.
\]
Thus $u$ and $v$ are the roots of
\[
z^2-\delta z+(\delta^2-1)/3=0.
\]
The discriminant is $(4-\delta^2)/3$, so real solutions exist exactly
when $|\delta|\le2$, equivalently $c\ge1/2$.  Solving the quadratic
and substituting $\delta=2(c-1)/c$ gives
\[
u,v
=\frac{c-1\pm\sqrt{(2c-1)/3}}{c},
\]
which proves the stated formula.
\end{proof}

\begin{remark*}
 For an arbitrary probability measure of full support on three states,
the same procedure reduces the classification to a cubic equation
depending on all three masses. 
\end{remark*}

\subsection{Generator classification with
\texorpdfstring{$E_{-1}$}{E(-1)}  spanned by the odd function}

We now turn the moment classification with
\texorpdfstring{$E_{-1}$}{E(-1)}  spanned by the odd function into a classification of
generators. 

\begin{proposition}
\label{prop:general-three-state-family}
Let $0<c<1$ and consider any irreducible $\nu_c$-reversible Markov generator $Q$
with $\lambda(Q)=1$. Then the equality
\[
E_{-1}(Q)=\operatorname{span}\{h_{\mathrm{odd}}\}
\] holds
if and only if $Q=Q_{a,c}$ for a unique parameter
$1<a\le 1/(1-c)$, where
\begin{equation}\label{eq:Q-a-c}
Q_{a,c}
=
\begin{pmatrix}
-ac&ac/2&ac/2\\[6pt]
a(1-c)&-[1+a(1-c)]/2&
          [1-a(1-c)]/2\\[6pt]
a(1-c)&[1-a(1-c)]/2&
          -[1+a(1-c)]/2
\end{pmatrix}.
\end{equation}
Every generator in this family is TMV.
\end{proposition}

%
%\begin{figure}[htbp]
%\centering
%\begin{tikzpicture}[>=Latex,every node/.style={font=\small}]
%\coordinate (m) at (-3.5,0);
%\coordinate (z) at (0,3.15);
%\coordinate (p) at (3.5,0);
%\fill (m) circle (2.5pt);
%\fill (z) circle (2.5pt);
%\fill (p) circle (2.5pt);
%\node[below=0.18cm] at (m) {$-1$};
%\node[above=0.18cm] at (z) {$0$};
%\node[below=0.18cm] at (p) {$1$};
%
%\draw[shorten <=0.08cm,shorten >=0.08cm,
%      -{Stealth[length=1.8mm,width=1.2mm]},bend left=18]
%  (m) to (z);
%\draw[shorten <=0.08cm,shorten >=0.08cm,
%      -{Stealth[length=1.8mm,width=1.2mm]},bend left=18]
%  (z) to (m);
%\draw[shorten <=0.08cm,shorten >=0.08cm,
%      -{Stealth[length=1.8mm,width=1.2mm]},bend right=18]
%  (p) to (z);
%\draw[shorten <=0.08cm,shorten >=0.08cm,
%      -{Stealth[length=1.8mm,width=1.2mm]},bend right=18]
%  (z) to (p);
%\draw[shorten <=0.08cm,shorten >=0.08cm,
%      -{Stealth[length=1.8mm,width=1.2mm]},bend left=18]
%  (m) to (p);
%\draw[shorten <=0.08cm,shorten >=0.08cm,
%      -{Stealth[length=1.8mm,width=1.2mm]},bend left=18]
%  (p) to (m);
%
%\node at (-2.62,2.12) {$a(1-c)$};
%\node at (-1.64,1.55) {$ac/2$};
%\node at (2.62,2.12) {$a(1-c)$};
%\node at (1.55,1.55) {$ac/2$};
%\node at (0,0.32) {$[1-a(1-c)]/2$};
%\node at (0,-0.89) {$[1-a(1-c)]/2$};
%\end{tikzpicture}
%\caption{The Markov chain generated by $Q_{a,c}$.  The direct jumps
%between the outer states disappear on the boundary
%$a(1-c)=1$.}
%\label{fig:general-three-state-chain}
%\end{figure}

\begin{figure}[htbp]
\centering
\includegraphics[scale=1]{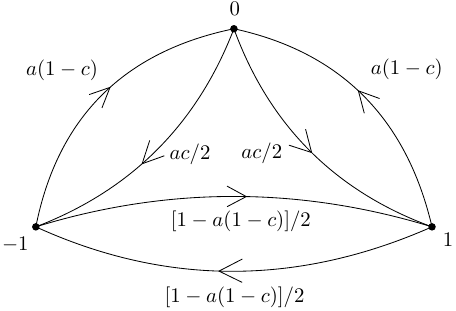}
\caption{The Markov chain generated by $Q_{a,c}$.  The direct jumps
between the outer states disappear on the boundary
$a(1-c)=1$.}
\label{fig:general-three-state-chain}
\end{figure}

Let us explain how the matrix \eqref{eq:Q-a-c} is found.  The
constant eigenfunction and the prescribed odd eigenfunction are
\begin{align}\label{e-01}
e_0=\one=(1,1,1)^\top,
\qquad
e_1=h_{\mathrm{odd}}=(0,1,-1)^\top.
\end{align}
To find the remaining eigenvector, write $e_2=(x,y,z)^\top$.  The
conditions $e_2\perp e_1$ and $e_2\perp e_0$ in $L^2(\nu_c)$ are
\[
\frac c2(y-z)=0,
\qquad
(1-c)x+\frac c2(y+z)=0.
\]
The first identity gives $y=z$.  Since an eigenvector may be
rescaled, we set $y=z=1$; the second identity then gives
$x=-c/(1-c)$.  Hence
\begin{align}\label{e-2}
e_2=(- c/(1-c),1,1)^\top.
\end{align}
The three vectors are mutually orthogonal in $L^2(\nu_c)$ and satisfy
\[
\|e_0\|_2^2=1,
\qquad
\|e_1\|_2^2=c,
\qquad
\|e_2\|_2^2=c/(1-c).
\]
We assign the eigenvalues $0,-1,-a$ to $e_0,e_1,e_2$, respectively.
The restriction $a>1$ is essential for the TMV property.  If $a=1$,
then
\[
E_{-1}(Q)=\operatorname{span}\{e_1,e_2\},
\]
and
$\E_{\nu_c}(e_1^2e_2)=c\ne0$; hence the cubic form
$h\mapsto\E_{\nu_c}h^3$ cannot vanish identically on
$E_{-1}(Q)$.

\begin{proof}[Proof of Proposition \ref{prop:general-three-state-family}]
Let $\mathsf E=[e_0,e_1,e_2]$ be the $3\times3$ matrix whose columns
are $e_0,e_1,e_2$.  Suppose first that $Q$ has the properties in the
statement.  Then $Q$ is self-adjoint on $L^2(\nu_c)$
and
\[
Qe_0=0,\qquad Qe_1=-e_1.
\]
Since 
$\operatorname{span}\{e_0,e_1\}^\perp = \mathbb{R} e_2$, there is $a>0$ such that $Qe_2=-ae_2$.  The normalization
$\lambda(Q)=1$ and the identity
$E_{-1}(Q)=\operatorname{span}\{e_1\}$ imply $a>1$.  Consequently,
\[
Q
=\mathsf E\operatorname{diag}(0,-1,-a)\mathsf E^{-1},
\]
and direct computation gives exactly \eqref{eq:Q-a-c}.  The
off-diagonal entries then imply $a\le1/(1-c)$.  The parameter $a$ is
unique because it is the remaining nonzero eigenvalue of $-Q$.

Conversely, let $1<a\le1/(1-c)$. Then it is easy to verify that   $Q_{a,c}$  satisfies all  the desired conditions.  This completes the proof.
\end{proof}

\section{Local phase transition}
\label{sec:defect}

 We
now pass from that classification to the hypercontractivity problem and we shall focus on the boundary case $a(1-c)=1$ of Proposition~\ref{prop:general-three-state-family} for  the TMV Markov generators.  

The main goal of this section is to prove that,  due to the sign-change of the discriminant of a quadratic polynomial, a local phase transition appears in the sharp hypercontractivity relation.

\subsection{The boundary  case of  TMV chains: simplified notation}
For the boundary case of Proposition~\ref{prop:general-three-state-family} for the TMV Markov generators, one has 
\[
c=1- 1/a,
\qquad a>1. 
\]
For notational simplicity, write
$$
Q_a:= Q_{a, c} = Q_{a, 1-1/a}, \qquad  \mu_a:=\nu_c = \nu_{1-1/a}.
$$

More explicitly, one has 
\begin{equation}\label{eq:Q-a}
Q_a= 
\begin{pmatrix}
-(a-1)&(a-1)/2&(a-1)/2\\[5pt]
1&-1&0\\
1&0&-1
\end{pmatrix},
\end{equation}
with invariant measure $\mu_a$ given by 
\begin{equation}\label{eq:mu-a}
\mu_a(0)=1/a,
\qquad
\mu_a(-1)=\mu_a(1)=(a-1)/(2a).
\end{equation}
By direct computation, the corresponding semigroup is
\begin{equation}\label{eq:P-a}
P_t^{(a)} 
= e^{t Q_a} = 
\begin{pmatrix}
\dfrac{1+(a-1)e^{-at}}a&
\dfrac{a-1}{2a}(1-e^{-at})&
\dfrac{a-1}{2a}(1-e^{-at})\\[8pt]
\dfrac{1-e^{-at}}a&
\dfrac{a-1+ae^{-t}+e^{-at}}{2a}&
\dfrac{a-1-ae^{-t}+e^{-at}}{2a}\\[8pt]
\dfrac{1-e^{-at}}a&
\dfrac{a-1-ae^{-t}+e^{-at}}{2a}&
\dfrac{a-1+ae^{-t}+e^{-at}}{2a}
\end{pmatrix}.
\end{equation}

\subsection{Definition of the defect function}
\label{sec:defi-defect}

Let $1<p<2$. 
Then the $(p,2)$-hypercontractivity problem at 
\[
t = \tau(p,2)=  \frac{1}{2}\log\frac{1}{p-1}
\] 
is equivalent to the problem of the non-negativity of the following functional
\[
f\mapsto \|f\|_p^2-\| P_{\tau(p,2)}^{(a)} f\|_2^2. 
\]
By standard argument (positivity $|P_tf|\le P_t|f|$, homogeneity, and continuity), one may consider only the functions $f$ of the form: 
\begin{align}\label{normal-form}
f= (1, \xi, \eta)^\top, \qquad \xi, \eta\ge 0. 
\end{align}

Therefore, by defining the defect function as follows
\[
\mathcal{D}_{a,p}(\xi,\eta)
:=\|(1, \xi, \eta)^\top \|_p^2-\|P_{\tau(p,2)}^{(a)} (1, \xi, \eta)^\top \|_2^2, \qquad \xi, \eta \ge 0, 
\]
one has 
\begin{align}\label{HC-def}
t_{\mathrm{opt}}^{(Q_a)} (p,2) = \tau(p,2) 
\quad\Longleftrightarrow\quad
\mathcal{D}_{a,p}(\xi,\eta)\ge0
\text{ for all }\xi,\eta\ge0.
\end{align}

\subsection{The local phase transition determined by a discriminant}

As usual, a necessary condition for hypercontractivity can be derived from small perturbations around the constant function. In the current situation,  the perturbation method gives rise to a local phase transition which is determined by a discriminant of a quadratic polynomial.

Recall that by \eqref{e-01} and \eqref{e-2}, the three eigenvectors (not necessarily normalized) of $Q_a$ are 
\[
e_0=\one=(1,1,1)^\top,
\qquad
e_1=(0,1,-1)^\top,  \qquad e_2=(- c/(1-c),1,1)^\top.
\]
Therefore,  a small perturbation of $\one = (1,1,1)^\top$ can be written as 
\[
e_0 + x e_1  + y e_2 = (1 - yc/(1-c), 1 + x + y, 1 -x +y)^\top. 
\]
Normalizing it to the form \eqref{normal-form} yields a small perturbation $(1, \xi, \eta)^\top$ of $\one$ with 
\[
\xi=1+w-v,
\qquad
\eta=1+w+v.
\]

\begin{lemma}
\label{lem:fourth-order-expansion}
As $(w,v)\to(0,0)$,
\begin{align}\label{eq:fourth-order-expansion}
\mathcal{D}_{a,p}(1+w-v,1+w+v)
=\frac{(a-1)(p-1)}{a^2}P_{a,p}(w,v^2) 
+o(w^2+v^4),
\end{align}
where $P_{a,p}$ is a quadratic form given by 
\begin{align}\label{quad-form}
P_{a,p}(X,Y) = (1-(p-1)^{a-1}) X^2-(2-p)X Y+ \frac{(2-p)(2ap-3p+3)}{12}Y^2.
\end{align}
\end{lemma}

\begin{proof}
A direct computation gives 
\begin{align*}
\mathcal{D}_{a,p}(\xi,\eta)
={}&
\left(\frac{2+(a-1)(\xi^p+\eta^p)}{2a}\right)^{2/p}
-
\left(\frac{2+(a-1)(\xi+\eta)}{2a}\right)^2
\\
&-\frac{a-1}{4a}(p-1)(\xi-\eta)^2
-\frac{a-1}{4a^2}(p-1)^a(\xi+\eta-2)^2.
\end{align*}
Expanding $\mathcal{D}_{a,p}(1+w-v,1+w+v)$ at $(w,v)=(0,0)$ gives 
\eqref{eq:fourth-order-expansion}.
\end{proof}

The discriminant of the quadratic form $P_{a,p}$ defined in \eqref{quad-form} is given by 
\begin{equation}
\label{eq:Delta}
\begin{split}
\Delta(a,p)
&=(2-p)^2 - \frac{\bigl(1-(p-1)^{a-1}\bigr)(2-p)(2ap-3p+3)}{3} \\
&= \frac{(2-p)(2ap-3p+3)}{3}
\left(
(p-1)^{a-1}
-\frac{2ap-3}{2ap-3p+3}
\right).
\end{split}
\end{equation}

\begin{proposition}
\label{prop:local-criterion}
Let $a>1$ and $1<p<2$.
\begin{enumerate}
\item If $\Delta(a,p)>0$, then 
\begin{align}\label{D-neg}
\inf_{(\xi, \eta)\in [0, \infty)^2}\mathcal{D}_{a,p} (\xi, \eta)<0. 
\end{align}
\item If $\Delta(a,p)<0$, then there is a neighborhood $U$ of $(1,1)$ such that for all $(\xi, \eta)\in U\setminus\{(1,1)\}$, 
\[
\mathcal{D}_{a,p}(\xi,\eta)>0.
\]
\end{enumerate}
\end{proposition}

\begin{proof}
Clearly, under the assumption  $a>1$ and $1<p<2$, one has 
\[
\frac{(a-1)(p-1)}{a^2}>0, \quad  
1-(p-1)^{a-1}>0. 
\]

Consequently, if $\Delta(a,p)>0$, 
then for the quadratic polynomial $P_{a,p}(X,1)$ in $X$, there exists $x_0\in \mathbb{R}$ such that $P_{a,p}(x_0,1) <0$. 
 Thus, taking $w = x_0 \varepsilon^2, v = \varepsilon$  in
\eqref{eq:fourth-order-expansion}, one obtains 
\begin{align*}\mathcal{D}_{a,p}(1 + x_0 \varepsilon^2 - \varepsilon, 1+ x_0\varepsilon^2 +\varepsilon) 
=  \frac{(a-1)(p-1)}{a^2} P_{a,p}(x_0,1) \varepsilon^4
+o( \varepsilon^4), 
\end{align*}
which implies the desired inequality \eqref{D-neg}.

If $\Delta(a,p)<0$, then the quadratic form $P_{a,p}(X,Y)$ is positive definite. Hence there exists a constant $c_{a,p}>0$ such that 
$$\text{$P_{a,p}(X,Y)\geq c_{a,p} (X^2+Y^2)$ for all $X,Y\in\mathbb{R}$.}$$
Taking $X=w$ and $Y=v^2$, we obtain 
\[
P_{a,p}(w,v^2) \geq c_{a,p}(w^2+v^4). 
\]
The assertion then follows. 
\end{proof}

\begin{corollary}
\label{cor-non-classical}
Let $a>1$ and $1<p<2$. Then the following implication holds: 
$$
\Delta(a,p)>0 \Longrightarrow t_{\mathrm{opt}}^{(Q_a)}(p,2)>\tau(p,2).
$$
\end{corollary}

\begin{proof}
The proof follows from the universal lower bound \eqref{eq:spectral-lower-bound},  the equivalence relation  \eqref{HC-def} and the item (1) of Proposition~\ref{prop:local-criterion} 
\end{proof}

\subsection{No phase transition for $1<a\le 3/2$}

\begin{proposition}
\label{prop-nopase-3/2}
Let $1<a\le 3/2$ and $1<p<q<\infty$. Then 
\begin{equation}
\label{eq:strict-from-Delta-pq}
t_{\mathrm{opt}}^{(Q_a)}(p,q)>\tau(p,q).
\end{equation}
\end{proposition}

\begin{proof}

For $\xi,\eta\ge0$, define
\[
\mathcal{D}_{a,p,q}(\xi,\eta)
:=
\bigl\|(1,\xi,\eta)^\top\bigr\|_p^q
-
\bigl\|P_{\tau(p,q)}^{(a)}
      (1,\xi,\eta)^\top\bigr\|_q^q.
\]
We take $q$-th powers here to maintain consistency with the definition of $\mathcal{D}_{a,p}$ above.
Put 
\[
r:=e^{-\tau(p,q)}
=\left(\frac{p-1}{q-1}\right)^{1/2}.
\]
By \eqref{eq:mu-a} and \eqref{eq:P-a}, direct computation gives 
\begin{align*}
\mathcal{D}_{a,p,q}(\xi,\eta)
={}&
\left(
\frac{2+(a-1)(\xi^p+\eta^p)}{2a}
\right)^{q/p}
-
\frac1a
\left(
\frac{1+(a-1)r^a}{a}
+\frac{(a-1)(1-r^a)}{2a}(\xi+\eta)
\right)^q\\
&
-\frac{a-1}{2a}
\left(
\frac{1-r^a}{a}
+\frac{a-1+ar+r^a}{2a}\xi
+\frac{a-1-ar+r^a}{2a}\eta
\right)^q\\
&
-\frac{a-1}{2a}
\left(
\frac{1-r^a}{a}
+\frac{a-1-ar+r^a}{2a}\xi
+\frac{a-1+ar+r^a}{2a}\eta
\right)^q.
\end{align*}
Write
\[
\xi=1+w-v,
\qquad
\eta=1+w+v.
\]
A Taylor expansion at $(w,v)=(0,0)$ gives
\begin{align}
\label{eq:general-pq-fourth-order-expansion}
\mathcal{D}_{a,p,q}(1+w-v,1+w+v)
=\frac{q(a-1)(p-1)}{2a^2} P_{a,p,q}(w,v^2)
+o(w^2+v^4), 
\end{align}
where 
\begin{equation}
\label{eq:P-apq}
P_{a,p,q}(X,Y) := (1-r^{2a-2}) X^2+(p-2-(q-2)r^a)X Y+ \frac{q-p}{12}\left((2a-3)(p-1)+\frac{2a}{q-1}\right)Y^2.
\end{equation}
By Lemma~\ref{lemma-dis-posi}, when $1<a\le 3/2$, the discriminant of the quadratic polynomial $P_{a,p,q}(X,1)$ in $X$ is positive. 
Moreover, since $1-r^{2(a-1)}>0$, 
there exists $x_0\in \mathbb{R}$ such that $P_{a,p,q}(x_0,1) <0$. 
Taking $w = x_0 \varepsilon^2, v = \varepsilon$ in
\eqref{eq:general-pq-fourth-order-expansion} yields 
\begin{align*}
\mathcal{D}_{a,p,q}(1 + x_0 \varepsilon^2 - \varepsilon, 1+ x_0\varepsilon^2 +\varepsilon) 
=  \frac{q(a-1)(p-1)}{2a^2} P_{a,p,q}(x_0,1) \varepsilon^4
+o( \varepsilon^4),  
\end{align*}
which implies 
\[
\mathcal{D}_{a,p,q}(1 + x_0 \varepsilon^2 - \varepsilon, 1+ x_0\varepsilon^2 +\varepsilon) <0
\] for all sufficiently small nonzero $\varepsilon$. 
Thus the hypercontractive inequality fails at
$t=\tau(p,q)$, and consequently \eqref{eq:strict-from-Delta-pq} holds. 
\end{proof}

We now prove that the discriminant is strictly positive when
\[
1<a\le\frac32.
\]

\begin{lemma}
\label{lemma-dis-posi}
Let $1<p<q$, $1<a\leq 3/2$, and $r=[(p-1)/(q-1)]^{1/2}$. 
Then the discriminant 
$$
\Delta(a,p,q) := \Big((p-2)-(q-2)r^a\Big)^2
-\frac{q-p}{3}
\left(
(2a-3)(p-1)+\frac{2a}{q-1}
\right)(1-r^{2a-2})
$$ of the quadratic form $P_{a,p,q}(X,Y)$ in \eqref{eq:P-apq} is positive.
\end{lemma}

\begin{proof}
Since $1<p<q$ and $1<a\leq3/2$, we have $0<r<1$ and 
$$-\frac{q-p}{3}(2a-3)(p-1)(1-r^{2a-2})\geq0.$$
It follows that 
$$\Delta(a,p,q)\geq \Big((p-2)-(q-2)r^a\Big)^2
-\frac{2a(q-p)}{3(q-1)}
(1-r^{2a-2}).$$
Using $p=1+(q-1)r^2$ to eliminate $p$ yields  
\begin{align*}
\Delta(a,p,q)&\geq\left[(1-r^a)+(q-1)(r^a-r^2)\right]^2
-\frac{2a}{3}(1-r^2)(1-r^{2a-2})\\
&=
(1-r^a)^2-\frac{2a}{3}(1-r^2)(1-r^{2a-2})
+2(q-1)(1-r^a)(r^a-r^2)
+(q-1)^2(r^a-r^2)^2\\
&>
(1-r^a)^2-\frac{2a}{3}(1-r^2)(1-r^{2a-2})\\
&\geq (1-r^a)^2 - (1-r^2)(1-r^{2a-2}) 
 = (r-r^{a-1})^2>0. 
\end{align*}
Here we use $0<r<1$ and $1<a\leq3/2$ again. 
This completes the proof.
\end{proof}

\section{Elementary analysis of the sign changes of $\Delta(a,p)$}

From Proposition~\ref{prop:local-criterion}, we see that the sign of  $\Delta(a,p)$ defined in \eqref{eq:Delta} plays an important role in our analysis.   This  section is  devoted to the analysis of the sign-changing of $\Delta(a,p)$.  

\begin{proposition}[See Figure \ref{fig:implicit-phase-diagram} for a quick view]\label{prop:sign-analysis}
We have the following three cases. 
\begin{itemize}
\item $1<a\le 3/2$: small $a$ implies positive discriminant. That is,   
\[
(a, p)\in (1, 3/2]\times (1,2) \Longrightarrow \Delta(a,p)>0.
\]
\item $a\ge 7/4$:  large $a$ implies negative discriminant. That is, 
\[
(a, p)\in [7/4, \infty) \times (1,2) \Longrightarrow \Delta(a,p)<0.
\]
\item $3/2< a< 7/4$: intermediate $a$ implies sign-changing discriminant. That is,  the equation $\Delta(a,p)=0$ has a unique solution $p_a\in (1,2)$
and 
\[
\Delta(a,p)<0\quad(1<p<p_a),
\qquad
\Delta(a,p)>0\quad(p_a<p<2).
\]
\end{itemize}
\end{proposition}

\begin{proposition}\label{prop-one-one}
 The  equation $\Delta(a,p_a)=0$ is equivalent to 
 \begin{align}\label{exp-eq}
 (p_a-1)^{a-1}
=\frac{2ap_a-3}
       {2ap_a-3(p_a-1)}.
 \end{align}
Moreover, the map 
$(3/2, 7/4)\to (1, 2)$, $a\mapsto p_a$
determined by $\Delta(a,p_a)=0$ is bijective. 
\end{proposition}

\subsection{Monotonicity of $a\mapsto \Delta(a,p)$ and the signs of $\Delta(a,p)$ at special points}
\begin{lemma}
\label{lem:mon-a}
Let $1<p<2$. Then the function $a\mapsto \Delta(a,p)$ 
is strictly decreasing on $(1,\infty)$. 
\end{lemma}

\begin{proof}
Recall the definition \eqref{eq:Delta} for $\Delta(a,p)$. A direct computation gives  
\begin{align*}
\partial_a\Delta(a,p)
=
-\frac{2-p}{3}
\Bigl[
2p\bigl(1-(p-1)^{a-1}\bigr)
-(2ap-3p+3)(p-1)^{a-1}\log(p-1)
\Bigr].
\end{align*}
Clearly, for $a>1$ and $1<p<2$, one has 
\[
0< (p-1)^{a-1}<1,
\qquad
\log(p-1)<0,\qquad
2ap-3p+3
=
2p(a-1)+(3-p)>0.
\]
Hence $\partial_a\Delta(a,p)<0$ and we complete the proof. 
\end{proof}

\begin{lemma}\label{lem-special}
Let $1<p<2$.  Then 
$
\Delta(3/2, p)>0$  and $\Delta(7/4, p)<0$. 
\end{lemma}

\begin{proof}
For $a = 3/2$, direct computation gives 
\begin{align*}
\Delta\left(3/2,p\right)
=
(2-p)^2\left(1-\frac{1-\sqrt{p-1}}{2-p}\right)
=
(2-p)^2\frac{\sqrt{p-1}}{1+\sqrt{p-1}}
>0.
\end{align*}
And for $a=7/4$, one has 
\begin{align*}
\Delta(7/4, p) = (2-p)^2\left(1-\frac{\bigl(1-(p-1)^{3/4}\bigr)(p/2+3)}{3(2-p)}\right)
 = -\frac{2-p}{6}\Big[7 p-6 -(p-1)^{3/4} (p+6)\Big]. 
\end{align*}
Then proving $\Delta(7/4,p)<0$
is equivalent to proving
\[
7 p-6 -(p-1)^{3/4} (p+6)>0. 
\]
Writing  $\theta=p-1\in(0,1)$, it suffices to prove
\[
(1+7\theta)^4-\theta^3(7+\theta)^4
=
(1-\theta)^3
\bigl(\theta^4+31\theta^3+384\theta^2+31\theta+1\bigr)>0,
\]
which is obvious.
\end{proof}

\subsection{Proof of Proposition~\ref{prop:sign-analysis}}

{\flushleft Case 1}. If $1<a\le 3/2$,  then by Lemmas~\ref{lem:mon-a} and~\ref{lem-special}, for any $1<p<2$,  one has 
\[
\Delta(a,p) \ge \Delta(3/2, p)>0. 
\]

{\flushleft Case 2}.  If $a\ge 7/4$,  then by Lemmas~\ref{lem:mon-a} and~\ref{lem-special}, for any $1<p<2$,  one has 
\[
\Delta(a,p)\le \Delta(7/4, p)<0. 
\]

{\flushleft Case 3}. If  $3/2<a<7/4$, then the last assertion of Proposition~\ref{prop:sign-analysis} is given in Lemma~\ref{lem:unique-p} below.

\begin{lemma}
\label{lem:unique-p}
For every $a\in(3/2,7/4)$, there exists a unique
$p_a\in(1,2)$ such that
$
\Delta(a,p_a)=0
$
and 
\[
\Delta(a,p)<0\quad(1<p<p_a),
\qquad
\Delta(a,p)>0\quad(p_a<p<2).
\]
\end{lemma}

\begin{proof}
By the formula \eqref{eq:Delta}, for $a\in(3/2,7/4)$ and $1<p<2$, the sign of $\Delta(a,p)$ is determined by the sign of 
\[
(p-1)^{a-1}
-\frac{2ap-3}{2ap-3p+3}. 
\]
This leads to the definition 
\[
H_a(p)
:=
(a-1)\log(p-1)
-\log\frac{2ap-3}{2ap-3p+3}
\]
and the lemma follows from the following claims
\begin{itemize}
\item Claim (i): $
\lim_{p\to1^{+}}H_a(p)=-\infty$ and $
H_a(2)=0$;
\item Claim (ii): $H_a$ has exactly one zero in $(1,2)$. 
\end{itemize}

Claim (i) is elementary, hence  it remains to prove Claim (ii). 

Direct computation gives 
\[
H_a'(p)
=
\underbrace{\frac{
a
}{
(p-1)(2ap-3)(2ap-3p+3)
} }_{\text{denoted $I_1(a,p)$}} \cdot \underbrace{  \Big[ 2(a-1)(2a-3)p^2-3p+3\Big]}_{\text{denoted $I_2(a,p)$}}.
\]
We need to analyze the sign of $H_a'(p)$.  Clearly,  for $a\in(3/2,7/4)$, one has 
\begin{itemize}
\item $I_1(a,p)>0$ for any $p\in (1,2)$;
\item  there exists $\widetilde{p_a}\in (1,2)$ such that   \[
I_2(a,p)>0  \quad (1<p< \widetilde{p_a}),  \qquad  I_2(a, p)<0 \quad (\widetilde{p_a}<p<2).
\]
\end{itemize}
  Therefore,  for 
$a\in(3/2,7/4)$, 
\begin{align}\label{inc-dec}
  \text{$H_a$ is strictly increasing on $(1,\widetilde{p_a})$ and strictly decreasing on  $(\widetilde{p_a},2)$.
}
\end{align}
Thus, combining Claim (i) and \eqref{inc-dec}, we complete the proof of Claim (ii).  
\end{proof}

\subsection{Proof of Proposition \ref{prop-one-one}}
The equivalence between the equation $\Delta(a,p_a)=0$ and \eqref{exp-eq} is clear from the formula \eqref{eq:Delta} for $\Delta(a,p)$.  Now for any $p\in(1,2)$, by Lemma~\ref{lem:mon-a}, the function
$a\mapsto\Delta(a,p)$ is continuous and strictly decreasing on $(1, \infty)$, while
Lemma~\ref{lem-special} gives
\[
\Delta(3/2,p)>0,
\qquad
\Delta(7/4,p)<0.
\]
Consequently there is a unique $a\in(3/2,7/4)$ for which
$\Delta(a,p)=0$.  On the other hand, Lemma~\ref{lem:unique-p} says
that, for each $a\in(3/2,7/4)$, there is a unique
$p_a\in(1,2)$ with $\Delta(a,p_a)=0$.  These two uniqueness
statements show that $a\mapsto p_a$ is a bijection from $(3/2,7/4)$ onto $(1,2)$.

\section{From local  to global}
\label{sec:threshold-curve}
This section is devoted to lifting the local hypercontractive phase transition obtained in Proposition~\ref{prop:local-criterion} to  the global hypercontractive phase transition.

\begin{proposition}\label{coro-haodeya}
Let $3/2<a<7/4$ and $1<p<2$. Then 
\[
\Delta(a,p)\le0
\quad\Longleftrightarrow\quad
\mathcal{D}_{a,p}(\xi,\eta)\ge0
\quad(\xi,\eta\ge0).
\]
\end{proposition}

\begin{corollary}\label{cor-global}
Let $3/2<a<7/4$ and $1<p<2$.  Let $p_a\in (1,2)$ be the unique solution of 
\[
(p_a-1)^{a-1}
=\frac{2ap_a-3}
       {2ap_a-3(p_a-1)}.
       \]
       Then 
\[
t_{\mathrm{opt}}^{(Q_a)}(p,2)
=\frac12\log\frac1{p-1}
\quad\Longleftrightarrow\quad 1<p\le p_a.
\]
\end{corollary}

\subsection{Proof of Proposition~\ref{coro-haodeya}}
By the part (1) of Proposition~\ref{prop:local-criterion}, it suffices to prove 
\[
\Delta(a,p)\le0
\quad\Longrightarrow\quad
\mathcal{D}_{a,p}(\xi,\eta)\ge0
\quad(\xi,\eta\ge0).
\]

From now on, we always assume that $\Delta(a,p)\le 0$.  The remaining proof uses some
technical inequalities collected in
Appendix~\ref{app:global-minimum-inequalities}. 
For convenience, in what follows, we shall use both notations $p$ and $\theta$ with  
\[
\theta=p-1\in(0,1).
\]  

For $w,u,v\ge0$, define the homogeneous functionals
\begin{align*}
\mathcal P(w,u,v)
&:=\frac{2w^p+(a-1)(u^p+v^p)}{2a},
\\
\mathcal R(w,u,v)
&:=\left(\frac{2w+(a-1)(u+v)}{2a}\right)^2
 +\frac{a-1}{4a}\theta(u-v)^2
 +\frac{a-1}{4a^2}\theta^a(u+v-2w)^2.
\end{align*}
Then
$$
\mathcal{D}_{a,p}(\xi,\eta)
=\mathcal P(1,\xi,\eta)^{2/p}
 -\mathcal R(1,\xi,\eta).
$$
It is therefore enough to prove
\begin{equation}\label{eq:homogeneous-global-inequality}
\mathcal R(w,u,v)\le \mathcal P(w,u,v)^{2/p}
\qquad(w,u,v\ge0).
\end{equation}

By homogeneity, we consider the compact set
\[
\Sigma:=\{(w,u,v)\in[0,\infty)^3:\mathcal P(w,u,v)=1\}
\]
and the desired inequality \eqref{eq:homogeneous-global-inequality} is reduced to 
$$
\kappa:=\sup_{(w,u,v)\in \Sigma}\mathcal R(w,u,v) \le 1.
$$
Since $(1,1,1)\in \Sigma$ and $\mathcal{R}(1,1,1)=1$, we have $\kappa\geq 1$. Hence we shall actually prove $\kappa=1$.

Since $\Sigma$ is compact, the maximum of $\mathcal{R}$ on $\Sigma$ is attained.   Choose any maximizer $(w,u,v)\in\Sigma$. It suffices to show 
\begin{equation}
\label{max-kappa}
\kappa = \mathcal{R}(w, u,v)\leq 1. 
\end{equation}

{\flushleft \bf Claim I:} we have $w, u,v>0$.  
 
 Indeed,
if $w=0$, then
\[
\partial_w\mathcal R(0,u,v)
=\frac{a-1}{a^2}(1-\theta^a)(u+v)>0.
\]
%Replacing $(0,u,v)$ by $(\varepsilon,u,v)$ and renormalizing it back
%to $\Sigma$ can be made completely explicit.  
Put
\[
\rho_1(\varepsilon)
=\left(1-\frac{\varepsilon^p}{a}\right)^{1/p}=1-\frac{\varepsilon^p}{ap}
 +O(\varepsilon^{2p}).
\]
Since $(0,u,v)\in\Sigma$,  we have
$\mathcal P\bigl(\varepsilon,\rho_1(\varepsilon)u,
                 \rho_1(\varepsilon)v\bigr)=1$. 
Then 
\[
\mathcal R\bigl(\varepsilon,\rho_1(\varepsilon)u,
                 \rho_1(\varepsilon)v\bigr)-\mathcal R(0,u,v)
=\partial_w\mathcal R(0,u,v)\,\varepsilon
 +O(\varepsilon^p)>0
\]
for every sufficiently small $\varepsilon>0$.  Thus $(0,u,v)$ cannot
be a maximizer.  

Similarly, if $u=0$, then
$$
\partial_u\mathcal R(w,0,v)
=\frac{a-1}{2a^2}
\left[2(1-\theta^a)w
 +(a-1-a\theta+\theta^a)v\right]>0.
$$
Here the last coefficient is positive because
$\psi(t):=a-1-at+t^a$ satisfies $\psi(1)=0$ and
$\psi'(t)=a(t^{a-1}-1)<0$ on $(0,1)$. 
Put 
\[
\rho_2(\varepsilon)=\left(1-\frac{a-1}{2a}\varepsilon^p\right)^{1/p}=1-\frac{a-1}{2ap}\varepsilon^p+O(\varepsilon^{2p}).
\]
Since $(w,0,v)\in\Sigma$,  we have
$\mathcal P\bigl(\rho_2(\varepsilon)w,\varepsilon,
                 \rho_2(\varepsilon)v\bigr)=1$. Also, one has 
\[
\mathcal R\bigl(\rho_2(\varepsilon)w,\varepsilon,
                 \rho_2(\varepsilon)v\bigr)-\mathcal R(w,0,v)
=\partial_u\mathcal R(w,0,v)\,\varepsilon
 +O(\varepsilon^p)>0
\]
for every sufficiently small $\varepsilon>0$.
Thus $(w,0,v)$ cannot
be a maximizer.   

By symmetry, $(w,u,0)$  also cannot be a maximizer. 

This completes the proof of {\bf Claim I}. 

{\flushleft \bf Claim II:}  we have $u=v$. 

Assuming {\bf Claim II}, we obtain 
\[
\kappa = \mathcal R(w,u,u) =\left(\frac{w+(a-1)u}{a}\right)^2
 +\frac{a-1}{a^2}\theta^a(u-w)^2
\]
and 
\[
1=\mathcal P(w,u,u)^{2/p} =  \Big(\frac{w^p+(a-1)u^p}{a}\Big)^{2/p}. 
\]
Hence Lemma~\ref{lem:weighted-two-point}  gives the desired inequality \eqref{max-kappa}. 

It remains to prove {\bf Claim II}. Indeed,  the Lagrange multiplier theorem gives a number $\lambda$ with
\begin{align}\label{eq:lagrange}
\nabla\mathcal R(w,u,v)
=\lambda\nabla\mathcal P(w,u,v).
\end{align}
The functions $\mathcal R$ and $\mathcal P$ are homogeneous of
degrees $2$ and $p$, respectively.  Therefore, by Euler's identities for homogeneous functions,  one obtains 
\[
2\mathcal R(w,u,v)
=\lambda p\,\mathcal P(w,u,v).
\]
Since $\mathcal R(w,u,v)=\kappa$ and $\mathcal P(w,u,v)=1$, it follows
that $\lambda=2\kappa/p$.  Expanding the Lagrange equations \eqref{eq:lagrange} gives
\begin{align*}
\frac{2w+(a-1)(u+v)}{2a}
-\frac{a-1}{2a}\theta^a(u+v-2w)
&=\kappa w^\theta,\\
\frac{2w+(a-1)(u+v)}{2a}
+\frac{\theta}{2}(u-v)
+\frac{\theta^a}{2a}(u+v-2w)
&=\kappa u^\theta,\\
\frac{2w+(a-1)(u+v)}{2a}
-\frac{\theta}{2}(u-v)
+\frac{\theta^a}{2a}(u+v-2w)
&=\kappa v^\theta.
\end{align*}
Subtracting the last two equations, and then comparing the first
equation with their average, yields
\begin{align}
\theta(u-v)
&=\kappa(u^\theta-v^\theta),
\label{eq:Lagrange-difference}\\
\theta^a\left(w-\frac{u+v}{2}\right)
&=\kappa\left(
w^\theta-\frac{u^\theta+v^\theta}{2}
\right).
\label{eq:Lagrange-average}
\end{align}

Assume by contradiction that  $u\neq v$.  By symmetry, we may assume $u>v$
and put
\[
x:=\frac vu\in(0,1),
\qquad
y:=\frac wu>0.
\]
Equations \eqref{eq:Lagrange-difference} and
\eqref{eq:Lagrange-average} imply
\begin{align}\label{eq:sanjiaozhuanhuan-1}
\frac{1+x^\theta}{2}-y^\theta
=
\theta^{a-1}\frac{1-x^\theta}{1-x}
\left(\frac{1+x}{2}-y\right).
\end{align}
Moreover, the constraint $\mathcal P(w,u,v)=1$, together with
\eqref{eq:Lagrange-difference}, gives
\begin{align}\label{eq:sanjiaozhuanhuan-2}
\left(
\frac{2y^p+(a-1)(1+x^p)}{2a}
\right)^{(2-p)/p}
=
\frac{\theta(1-x)}
{\kappa(1-x^\theta)}.
\end{align}

Introduce the following change of variables 
\[
x=e^{-2s},
\qquad
t=e^s y,
\qquad s>0.
\]
Then  \eqref{eq:sanjiaozhuanhuan-1} and \eqref{eq:sanjiaozhuanhuan-2} become 
\begin{align}\label{eq:stationary-hyperbolic-relation-1}
\cosh(\theta s)-t^\theta
=
\theta^{a-1}
\frac{\sinh(\theta s)}{\sinh s}
(\cosh s-t),
\end{align}
\begin{align}\label{eq:stationary-hyperbolic-relation-2}
t^p+(a-1)\cosh(ps)
=
a\left(
\frac{\theta\sinh s}
{\kappa\sinh(\theta s)}
\right)^{p/(2-p)}.
\end{align}

On the one hand,  combining  \eqref{eq:stationary-hyperbolic-relation-2} and $\kappa\geq1$ with
Lemma~\ref{lem:Ltheta-bound}, we obtain 
\begin{align*}
t^p+(a-1)\cosh(ps)
=
a\left(
\frac{\theta\sinh s}
{\kappa\sinh(\theta s)}
\right)^{p/(2-p)} <a  \cdot \frac{\cosh(ps)+2}{3}, 
\end{align*}
which is equivalent to 
\begin{equation}\label{eq:t-upper-bound}
1-t^p>
\frac{2a-3}{3}\bigl(\cosh(ps)-1\bigr) = \frac{4a-6}{3} \sinh^2(ps/2) .
\end{equation}
Since $a>3/2$, 
\eqref{eq:t-upper-bound} implies  that $0<t<1$ and hence 
\begin{align}\label{st-rel}
\cosh(s) >t. 
\end{align}

On the other hand, using \eqref{st-rel} and the elementary equivalence 
\[
\Delta(a,p)\leq0 \Longleftrightarrow (p-1)^{a-1}\leq \frac{2ap-3}{2ap-3\theta},
\]
the equality \eqref{eq:stationary-hyperbolic-relation-1} implies
$$
\cosh(\theta s)-t^\theta
\leq   \frac{\sinh(\theta s)}{\sinh s} \frac{2ap-3}{2ap-3\theta} (\cosh s-t).
$$
That is 
\begin{align*}
2 \sinh^2(\theta s/2)+  1-t^\theta
\leq  \underbrace{ \frac{\sinh(\theta s)}{\sinh s} \frac{2ap-3}{2ap-3\theta}}_{\text{denoted $\gamma$}}  \cdot \Big[2 \sinh^2(s/2) +1-t\Big].
\end{align*}
However,  since $0< \sinh(\theta s)< \theta \sinh(s)$ (see Lemma~\ref{lem:hyperbolic-monotonicity}) and $0< 2ap-3<2ap - 3 \theta$, one has $0<\gamma<\theta$. Hence, by using the elementary inequalities 
\[
1-t^\theta>\theta(1-t),
\qquad
p(1-t)>1-t^p \quad \text{for all $0<t<1$ and $0<\theta<1<p$,}
\]
one obtains 
\begin{align*}
0&\geq  2 \sinh^2(\theta s/2)+  1-t^\theta  - \gamma \Big[ 2 \sinh^2(s/2) +1 -t\Big]
\\
 & = 2 \sinh^2(\theta s/2) - 2\gamma  \sinh^2(s/2) + 1 - t^\theta - \gamma(1-t)
\\
 &>  2 \sinh^2(\theta s/2) - 2\gamma  \sinh^2(s/2) + \theta(1-t) - \gamma ( 1  -t)
\\
 &=  2 \sinh^2(\theta s/2)  - 2\gamma  \sinh^2(s/2)  + (\theta-\gamma)(1-t) 
\\
& >  2 \sinh^2(\theta s/2)   - 2\gamma  \sinh^2(s/2) + (\theta-\gamma)\frac{1-t^p}{p}. 
\end{align*}
The above inequality combined with \eqref{eq:t-upper-bound} implies 
\begin{align*}
0&\geq  2 \sinh^2(\theta s/2)+  1-t^\theta  - \gamma \Big[ 2 \sinh^2(s/2) +1 -t\Big] 
\\
&>2 \sinh^2(\theta s/2)   - 2\gamma  \sinh^2(s/2) + \frac{\theta-\gamma}{p} \frac{4a-6}{3}  \sinh^2(ps/2).
\end{align*}
This contradicts Lemma~\ref{lem:terminal-hyperbolic}. Thus we complete the proof of {\bf Claim II}.

\subsection{Proof of Corollary \ref{cor-global}}
By Corollary~\ref{cor-non-classical},  Proposition~\ref{prop-one-one}, Lemma~\ref{lem:unique-p},  Proposition~\ref{coro-haodeya}, as well as the equivalence  relation \eqref{HC-def},   we obtain the desired equivalence. 

\section{Proofs of the main theorems}

\subsection{The complementary rigidity range}

\begin{proposition}
\label{thm:global-rigidity}
For every $a\ge7/4$, the Markov chain generated by $Q_a$ is
hypercontractively rigid: for every $1<p<q<\infty$,
\[
t_{\mathrm{opt}}^{(Q_a)}(p,q)
=\frac12\log\frac{q-1}{p-1}.
\]
Equivalently, $\rho_{\mathrm{LS}}(Q_a)=1$.
\end{proposition}

\begin{proof}
It suffices to show that $\rho_{\mathrm{LS}}(Q_a)=1$. 
In the notation of Chen, Liu, and Saloff-Coste
\cite[Section~4]{ChenLiuSaloffCoste2008}, the logarithmic Sobolev constant is defined by 
\[
\alpha(K)
:=\inf_{\operatorname{Ent}_\mu(f^2)>0}
\frac{\langle f,(I-K)f\rangle_{L^2(\mu)}}
     {\operatorname{Ent}_\mu(f^2)},
\]
with
\[\operatorname{Ent}_\mu(f^2)=\!\int\! f^2\log f^2\,d\mu
-\! \int \! f^2\,d\mu\cdot \log \! \int \! f^2 \,d\mu.\]
Thus, if $Q=K-I$, our convention gives
$
\rho_{\mathrm{LS}}(Q)=2\alpha(K).
$
Suppose first that $7/4\le a\le2$ and put $\vartheta=a-1$.  After
ordering the states as $(-1,0,1)$, one has
\[
I+Q_a
=
\begin{pmatrix}
0&1&0\\[2pt]
\vartheta/2&1-\vartheta&\vartheta/2\\[2pt]
0&1&0
\end{pmatrix}
=:K_\vartheta .
\]
Here $\vartheta\in[3/4,1]$.  By
\cite[Theorem~4.3]{ChenLiuSaloffCoste2008},
$\alpha(K_\vartheta)=1/2$, and hence
$
\rho_{\mathrm{LS}}(Q_a)=1.
$

Now assume that $a>2$.  Then 
$
r=(a-2)/(a-1)\in(0,1)$. Again in the order $(-1,0,1)$, define
\[
\widetilde K_r
:=
\begin{pmatrix}
r&1-r&0\\[2pt]
1/2&0&1/2\\[2pt]
0&1-r&r
\end{pmatrix}.
\]
A direct comparison with
\eqref{eq:Q-a} gives
\[
\widetilde K_r-I= \widetilde Q_r = \frac1{a-1}Q_a.
\]
The invariant measure of $\widetilde K_r$ is
\[
\left(
\frac1{4-2r},
\frac{2-2r}{4-2r},
\frac1{4-2r}
\right)
=
\left(
\frac{a-1}{2a},
\frac1a,
\frac{a-1}{2a}
\right),
\]
which is precisely $\mu_a$ in the reordered coordinates. 

By \cite[Theorem 4.1]{ChenLiuSaloffCoste2008}, we have 
$
\alpha(\widetilde K_r)= (1-r)/2
$
and hence 
\[
\rho_{\mathrm{LS}}(\widetilde Q_r)
=2\alpha(\widetilde K_r)
=1-r
=1/(a-1).
\]
Using the elementary equality
$
\rho_{\mathrm{LS}}(cQ)=c\,\rho_{\mathrm{LS}}(Q) $ 
and the relation  $Q_a=(a-1)\widetilde Q_r$, one obtains
$
\rho_{\mathrm{LS}}(Q_a)=1.
$
This completes the whole proof.
\end{proof}

\subsection{Proofs of  Theorem~\ref{thm:prescribed-local-threshold} and  Theorem~\ref{thm:family-phase-diagram}}

By using Proposition~\ref{prop-one-one} and duality, 
Theorem~\ref{thm:prescribed-local-threshold} follows from Theorem~\ref{thm:family-phase-diagram}. 

It remains to prove  Theorem~\ref{thm:family-phase-diagram}.  We distinguish the three parameter ranges appearing in
Proposition~\ref{prop:sign-analysis}.
\begin{itemize}
\item 
If $1<a\le3/2$, then  for every $1<p<q$, by Proposition~\ref{prop-nopase-3/2},  one has 
\[
t_{\mathrm{opt}}^{(Q_a)}(p,q)>\tau(p,q).
\]
\item 
If $3/2<a<7/4$, then  the assertion of Theorem~\ref{thm:family-phase-diagram} is obtained in Corollary~\ref{cor-global}. 
\item  If $a\ge7/4$, Proposition~\ref{thm:global-rigidity} shows that
the sharp hypercontractive relation holds for every exponent pair.
\end{itemize}

Therefore, the Markov generator   $Q_a$ exhibits a
hypercontractivity phase transition if and only if
$a\in (3/2, 7/4)$ and we complete the whole proof.

\appendix

\section{Some technical inequalities}
\label{app:global-minimum-inequalities}

\subsection{A two-point inequality}

\begin{lemma}
\label{lem:weighted-two-point}
Let $3/2\le a\le2$ and $1<p<2$.  For every $x,y\ge0$,
$$
\left(\frac{x^p}{a}+
\left(1-\frac1a\right)y^p\right)^{2/p}
\ge
\left(\frac xa+\left(1-\frac1a\right)y\right)^2
+\frac{a-1}{a^2}(p-1)^a(x-y)^2.
$$
\end{lemma}

\begin{proof}
By continuity, we just need to consider $x,y>0$. 
By homogeneity, set $y=1$ and define
\[
h(x):=
\left(\frac{x^p}{a}+1-\frac1a\right)^{2/p}
-\left(\frac xa+1-\frac1a\right)^2
-\frac{a-1}{a^2}(p-1)^a(x-1)^2.
\]
We need to prove 
\begin{equation}
\label{eq:h-pos}
\text{$h(x)\geq0$ for all $x>0$}. 
\end{equation}

{\flushleft \bf Claim A: } $h$ is convex on $(0, \infty)$. 

A direct computation gives $h(1)=h'(1)=0$. This combined with {\bf Claim A} yields the desired assertion \eqref{eq:h-pos}. 
  
  It remains to prove {\bf Claim A}.  Indeed, for $x>0$, a direct differentiation gives
\[
h''(x)=\frac2a\Big[ \underbrace{
(p-1) s^{2-p}
+\frac{2-p}{a} s^{2-2p}
-\frac1a-\frac{a-1}{a}(p-1)^a}_{\text{denoted $g(s)$}}
\Big] \quad   \text{\,\,with\,\,} s=\frac{1}{x} \left(\frac{x^p}{a}+1-\frac1a\right)^{1/p}.
\]
So it suffices to show that for $3/2\le a\le2$ and $1<p<2$, one has 
\[
\text{$g(s)\geq0$ for all $s>0$}. 
\]
To this purpose, note that  
\[
g'(s)=(p-1)(2-p)s^{1-p}
\left(1-\frac2a s^{-p}\right).
\]
Thus $g$ takes its minimum at $s=(2/a)^{1/p}$.  Hence it suffices to show that 
\begin{align}\label{ell-pos}
g((2/a)^{1/p}) = \underbrace{\frac p2\left(\frac2a\right)^{ \frac{2-p}{p} } - \Big[
\frac1a+\frac{a-1}{a}(p-1)^a\Big]}_{\text{denoted $\ell(p)$}}\geq 0. 
\end{align}

{\flushleft \bf Claim B:} $\ell$ is concave on $(1,2)$. 

Clearly,  $\ell(1)= \ell(2)=0$, which combined with {\bf Claim B} yields the desired inequality 
\eqref{ell-pos}. 

It remains to prove {\bf Claim B}. Indeed, one has 
\begin{align*}
\ell''(p)= \frac2{p^3}\left(\log\frac2a\right)^2 \cdot 
\left(\frac2a\right)^{\frac{2-p}{p}} - (a-1)^2(p-1)^{a-2}. 
\end{align*}
Now note that the assumption $3/2\le a\le2$ and $1<p<2$ implies that 
\[
(a-1)^2(p-1)^{a-2} >\frac14
\]
and (using the monotonicity on both $a$ and $p$, taking $a=3/2$ and $p=1$ gives the maximum)
 \[
\frac2{p^3}\left(\log\frac2a\right)^2
\left(\frac2a\right)^{\frac{2-p}{p}}
<\frac83\left(\log\frac43\right)^2<\frac14.
\]
Thus $\ell''(p)<0$ on $(1,2)$ and we complete the whole proof.
\end{proof}

\subsection{Some inequalities on hyperbolic functions}
 In this subsection, let $1<p<2$ and put 
 \[
 \theta=p-1\in(0,1).
 \]
 
The following lemma is classical and  follows from \cite[Theorem 3]{AVV}. 

\begin{lemma}
\label{lem:hyperbolic-monotonicity}
The function $s\mapsto\sinh s/s$ is strictly increasing and the
function $s\mapsto\tanh s/s$ is strictly decreasing on $(0,\infty)$.
In particular, for $s>0$,
$$
\frac{\sinh(\theta s)}{\sinh s}<\theta,
\qquad
\frac{\tanh(\theta s/2)}{\tanh (s/2)}>\theta, \qquad \frac{\sinh(ps/2)}{p\sinh(s/2)}>1,\qquad
\frac{\theta\sinh(ps/2)}{p\sinh(\theta s/2)}>1.
$$
\end{lemma}

\begin{lemma}
\label{lem:Ltheta-bound}
For any  $s> 0$, one has 
\begin{equation}
\label{eq:Ltheta-bound-appendix}
 \left(\frac{\theta\sinh s}{\sinh(\theta s)}
\right)^{p/(2-p)}  < \frac{\cosh(ps)+2}{3}.
\end{equation}
\end{lemma}

\begin{proof}
Recall that $\theta=p-1$. 
Let
\[
F(s)
:=
\log\frac{\cosh(ps)+2}{3}
-\frac{p}{2-p}
\log\frac{\theta\sinh s}{\sinh(\theta s)}.
\]
Clearly,
$
\lim_{s\to0^+}F(s)=0.
$
Therefore, it suffices to show that $F'(s)>0$ for any $s>0$.

Indeed, 
\[
\frac{F'(s)}{p}
=
\frac{\sinh(ps)}{\cosh(ps)+2}
-
\frac{\frac{\cosh s}{\sinh s}-\theta  \frac{\cosh(\theta s)}{\sinh(\theta s)}}{1-\theta}.
\]
Thus $F'(s)>0$ is equivalent to $N(s)>0$, where
\begin{align*}
N(s)
={}&
(1-\theta)\sinh(ps)\sinh s\,\sinh(\theta s)-
(\cosh(ps)+2)
\bigl(
\cosh s\,\sinh(\theta s)
-\theta\sinh s\,\cosh(\theta s)
\bigr).
\end{align*}
By elementary product-to-sum identities for hyperbolic cosine and sine functions, we obtain 
\[
2N(s)
=
\theta\sinh(2s)-\sinh(2\theta s)
+2p\sinh((1-\theta)s)
-2(1-\theta)\sinh(ps).
\]

We now expand $2N(s)$ into its power series. The
coefficient of $s^{2n+1}/(2n+1)!$ is zero for $n=0$, while for
$n\ge1$ it equals
\[
2(1-\theta^2)
\sum_{j=0}^{n-1}
\theta^{2j+1}
\left[
4^n-2\binom{2n}{2j+1}
\right].
\]
Every summand is nonnegative, since
\[
2\binom{2n}{2j+1}
\le
2\sum_{\substack{0\le \ell\le 2n\\ \ell\ \mathrm{odd}}}
\binom{2n}{\ell}
=
4^n.
\]
Moreover, the coefficient is strictly positive starting at order
$s^5$. Hence
$
N(s)>0,
$
as desired. 
\end{proof}

\begin{remark*}
The equality \eqref{eq:stationary-hyperbolic-relation-2} can be rewritten as 
$$
\left(\frac{\theta\sinh s}{\sinh(\theta s)}\right)^{p/(2-p)}
=\kappa^{p/(2-p)}
\frac{(a-1)\cosh(p s) + t^p}{a}. 
$$
Formally substituting $\kappa=t=1$ and $a=3/2$ in the right-hand side suggests the comparison function $(\cosh(p s)+2)/3$, i.e., the right-hand side of \eqref{eq:Ltheta-bound-appendix}. 
\end{remark*}

\begin{lemma}
\label{lem:terminal-hyperbolic}
Let $a>3/2$ and  $1<p<2$. For any $s>0$ put
\[
\gamma:=\frac{\sinh(\theta s)}{\sinh s}
\frac{2ap-3}{2ap-3\theta}.
\]
Then
$$
L: =  \sinh^2(\theta s/2)   - \gamma  \sinh^2(s/2) + \frac{\theta-\gamma}{p} \frac{2a-3}{3}  \sinh^2(ps/2)>0.
$$
\end{lemma}

\begin{proof}
Since $2ap-3\theta>0$ and $\theta=p-1$, multiplying $L$ by $2ap-3\theta$ and rearranging (expanding exactly as a quadratic polynomial in $2a-3$) gives
\begin{align*}
& (2ap-3\theta)\cdot L 
\\
={} & \underbrace{
3   \left[ \sinh^2(\theta s/2)  
-\theta\frac{\sinh(\theta s)}{\sinh s} \sinh^2(s/2) \right]}_{\text{denoted $\mathrm{I}$}} + \underbrace{
\frac{(2a-3)^2}{3}
\left(\theta-\frac{\sinh(\theta s)}{\sinh s}\right) \sinh^2(p s/2) }_{\text{denoted $\mathrm{II}$}}
\\
&+ 
\underbrace{(2a-3)\left\{
p\left[\sinh^2(\theta s/2)  
-\frac{\sinh(\theta s)}{\sinh s} \sinh^2(s/2)  \right]
+\frac{\theta}{p}
\left(1-\frac{\sinh(\theta s)}{\sinh s}\right) \sinh^2(p s/2)  
\right\}}_{\text{denoted $\mathrm{III}$}}. 
\end{align*}
It suffices to show that 
\[
\mathrm{I}>0, \quad \mathrm{II}>0, \quad  \mathrm{III}>0.
\]

For $\mathrm{I}$, since 
\[
\mathrm{I} = \frac{3}{2}   \sinh(\theta s) \tanh(s/2)    \left[ \frac{\tanh(\theta s/2) }{\tanh (s/2)}
-\theta \right],
\]
 by Lemma~\ref{lem:hyperbolic-monotonicity}, one has $\mathrm{I}>0$. 

Similarly, again  by Lemma~\ref{lem:hyperbolic-monotonicity}, one has $\mathrm{II}>0$. 

We now show $\mathrm{III}>0$.  Note that  (using $p=1+\theta$ and sum-to-product, product-to-sum formulae)
\begin{align*}
 \sinh^2\frac{\theta s}{2}
-\frac{\sinh(\theta s)}{\sinh s} \sinh^2\frac{s}{2}
=
-\frac{\sinh(\theta s/2)\sinh((1-\theta)s/2)}
{\cosh(s/2)},
\end{align*}
\begin{align*}
1-\frac{\sinh(\theta s)}{\sinh s}=
\frac{
\cosh(ps/2)\sinh((1-\theta)s/2)
}{
\sinh(s/2)\cosh(s/2)
}.
\end{align*} 
Therefore,
\begin{align*} 
\frac{\mathrm{III}}{2a-3} 
&=
-\frac{p\sinh(\theta s/2)\sinh((1-\theta)s/2)}{\cosh(s/2)}
+\frac{\theta\cosh(ps/2)\sinh((1-\theta)s/2) \sinh^2(ps/2)}{
p\sinh(s/2)\cosh(s/2)}\\
&=
\frac{p\sinh((1-\theta)s/2)\sinh(\theta s/2)}{\cosh(s/2)}
\left[
\frac{
\theta\sinh^2(ps/2)\cosh(ps/2)
}{
p^2\sinh(s/2)\sinh(\theta s/2)
}
-1
\right]\\
& = \frac{p\sinh((1-\theta)s/2)\sinh(\theta s/2)}{\cosh(s/2)}
\left[
\frac{\sinh(ps/2)}{p\sinh(s/2)} \cdot 
\frac{\theta\sinh(ps/2)}{p\sinh(\theta s/2)} \cdot 
\cosh(ps/2)-1
\right].
\end{align*}
Then, again by Lemma~\ref{lem:hyperbolic-monotonicity} and the inequality
$
\cosh(ps/2)>1,
$ one obtains $\mathrm{III}>0$. 

This completes the whole proof. 
\end{proof}

\section*{Acknowledgements}

\medskip
\noindent\textbf{Funding.}
Jie Cao is supported by NSFC (No.~12371075).  Shilei Fan is supported by NSFC (No.~12331004 and 12231013).  Yong Han is
supported by NSFC
(No.~12131016).  Yanqi Qiu is supported by NSFC (No.~12595283 and 12471145).  Zipeng Wang
is supported by NSFC
(No.~12471116) and the Fundamental Research Funds for the Central
Universities (No.~2025CDJ-IAIS YB004).

\medskip
\noindent\textbf{Declaration of competing interests.}
The authors declare no competing interests.

\medskip
\noindent\textbf{Disclosure of AI assistance}. The authors used ChatGPT (OpenAI, model GPT-5.6-sol) to
help simplify the proof of the nonnegativity of the defect function in Section~\ref{sec:threshold-curve}, after having
established a longer proof themselves. The model suggested a substantially shorter argument based
on standard homogenization and a change of variables involving hyperbolic functions, together
with manipulations of hyperbolic identities. The authors fully reworked the resulting argument
for clarity and, through further exchanges with the model, developed the concise formulation and
proofs of Lemma~\ref{lem:Ltheta-bound} and Lemma~\ref{lem:terminal-hyperbolic} in Appendix~\ref{app:global-minimum-inequalities}. The authors take full
responsibility for the mathematical correctness and final presentation of the paper.

\end{document}